\documentclass[12pt,a4paper]{amsart}

\usepackage[usenames]{color}
\usepackage[bookmarks=true,bookmarksnumbered=true,bookmarkstype=toc,colorlinks=true,linkcolor=black,citecolor=black]{hyperref}
\usepackage[active]{srcltx} 

\usepackage[abs]{overpic}

\newtheorem{thm}{Theorem}
\newtheorem{lem}[thm]{Lemma}
\newtheorem{cor}[thm]{Corollary}
\newtheorem{prop}[thm]{Proposition}

\theoremstyle{definition}
\newtheorem{deff}[thm]{Definition}
\newtheorem{example}[thm]{Example}
\theoremstyle{remark}
\newtheorem{rem}[thm]{Remark}

\newcommand{\refeq}[1]{\textup{(\ref{eq:#1})}}
\newcommand{\refthm}[1]{Theorem~\ref{thm:#1}}
\newcommand{\reflem}[1]{Lemma~\ref{lem:#1}}
\newcommand{\refcor}[1]{Corollary~\ref{cor:#1}}
\newcommand{\refprop}[1]{Proposition~\ref{prop:#1}}
\newcommand{\refsec}[1]{Section~\ref{sec:#1}}
\newcommand{\refexa}[1]{Example~\ref{exa:#1}}

\numberwithin{equation}{section}
\numberwithin{thm}{section}

\renewcommand{\theenumi}{\roman{enumi}}
\renewcommand{\labelenumi}{\textup{(\theenumi)}}
{\catcode`p =12 \catcode`t =12 \gdef\eeaa#1pt{#1}}      
\def\accentadjtext#1{\setbox0\hbox{$#1$}\kern   
                \expandafter\eeaa\the\fontdimen1\textfont1 \ht0 }
\def\accentadjscript#1{\setbox0\hbox{$#1$}\kern 
                \expandafter\eeaa\the\fontdimen1\scriptfont1 \ht0 }
\def\accentadjscriptscript#1{\setbox0\hbox{$#1$}\kern   
                \expandafter\eeaa\the\fontdimen1\scriptscriptfont1 \ht0 }
\def\accentadjtextback#1{\setbox0\hbox{$#1$}\kern       
                -\expandafter\eeaa\the\fontdimen1\textfont1 \ht0 }
\def\accentadjscriptback#1{\setbox0\hbox{$#1$}\kern     
                -\expandafter\eeaa\the\fontdimen1\scriptfont1 \ht0 }
\def\accentadjscriptscriptback#1{\setbox0\hbox{$#1$}\kern 
                -\expandafter\eeaa\the\fontdimen1\scriptscriptfont1 \ht0 }
\def\itoverline#1{{\mathsurround0pt\mathchoice
        {\rlap{$\accentadjtext{\displaystyle #1}
                \accentadjtext{\vrule height1.593pt}
                \overline{\phantom{\displaystyle #1}
                \accentadjtextback{\displaystyle #1}}$}{#1}}
        {\rlap{$\accentadjtext{\textstyle #1}
                \accentadjtext{\vrule height1.593pt}
                \overline{\phantom{\textstyle #1}
                \accentadjtextback{\textstyle #1}}$}{#1}}
        {\rlap{$\accentadjscript{\scriptstyle #1}
                \accentadjscript{\vrule height1.593pt}
                \overline{\phantom{\scriptstyle #1}
                \accentadjscriptback{\scriptstyle #1}}$}{#1}}
        {\rlap{$\accentadjscriptscript{\scriptscriptstyle #1}
                \accentadjscriptscript{\vrule height1.593pt}
                \overline{\phantom{\scriptscriptstyle #1}
                \accentadjscriptscriptback{\scriptscriptstyle #1}}$}{#1}}}}
\newcommand{\art}[6]{{\rm #1, \emph{#2}, #3\/ \bf #4 \rm (#5), \mbox{#6}.}}
\newcommand{\auth}[2]{{#2. #1}}
\newcommand{\AND}{{\rm and }}

\newcommand{\qtext}{\quad\text}
\newcommand{\sm}{\setminus}

\newcommand{\ve}{\varepsilon}
\newcommand{\bd}{\partial}         

\newcommand{\R}{{\mathbb R}}
\newcommand{\Z}{{\mathbb Z}}

\DeclareMathOperator{\dist}{dist}
\DeclareMathOperator{\diam}{diam}
\DeclareMathOperator{\BMOA}{BMOA}
\DeclareMathOperator*{\essliminf}{ess\,lim\,inf}
\DeclareMathOperator*{\essinf}{ess\,inf}

\newcommand{\increase}{\uparrow}
\newcommand{\union}{\bigcup}
\newcommand{\const}{C}

\newcommand{\di}{n}			    
\newcommand{\Rd}{{\R^\di}}		
\newcommand{\Dom}{\Omega}		    

\newcommand{\din}{d_{\textup{in}}}
\newcommand{\Bin}{B_{\textup{in}}}
\newcommand{\BX}{B_X}

\newcommand{\densname}{{\mathcal{D}}}
\newcommand{\densCpname}{{\mathcal{D}}^{\Cp}}
\newcommand{\densinname}{{\mathcal{D}}_{\textup{in}}}

\newcommand{\dens}[3]{\densname(#1,#2,#3)} 
\newcommand{\densCp}[2]{\densCpname(#1,#2)} 
\newcommand{\densin}[3]{\densinname(#1,#2,#3)} 

\newcommand{\dconst}{\tau}

\DeclareMathOperator{\capa}{cap}  
\newcommand{\capap}{\capa_{p}}
\newcommand{\capapRn}{\capa_{p}^{\R^n}}
\newcommand{\capapED}[2]{\capap(#1,#2)}
\newcommand{\capapEDRn}[2]{\capapRn(#1,#2)}
\newcommand{\UU}{U^*}       
\newcommand{\Jclass}{\mathcal{J}}

\newcommand{\disof}[1]{\delta_{#1}}    
\newcommand{\disU}{\disof{U}}    

\def\vint{\mathop{\mathchoice%
          {\setbox0\hbox{$\displaystyle\intop$}\kern 0.22\wd0%
           \vcenter{\hrule width 0.6\wd0}\kern -0.82\wd0}%
          {\setbox0\hbox{$\textstyle\intop$}\kern 0.2\wd0%
           \vcenter{\hrule width 0.6\wd0}\kern -0.8\wd0}%
          {\setbox0\hbox{$\scriptstyle\intop$}\kern 0.2\wd0%
           \vcenter{\hrule width 0.6\wd0}\kern -0.8\wd0}%
          {\setbox0\hbox{$\scriptscriptstyle\intop$}\kern 0.2\wd0%
           \vcenter{\hrule width 0.6\wd0}\kern -0.8\wd0}}%
          \mathopen{}\int}
\newcommand{\Om}{\Omega}
\newcommand{\om}{\omega}
\renewcommand{\phi}{\varphi}
\newcommand{\al}{\alpha}
\newcommand{\CC}{{\mathfrak{U}}}
\newcommand{\be}{\beta}
\newcommand{\eps}{\varepsilon}
\newcommand{\La}{\Lambda}
\newcommand{\de}{\delta}
\newcommand{\Ga}{\Gamma}
\newcommand{\ga}{\gamma}
\newcommand{\gat}{\tilde{\gamma}}
\newcommand{\ka}{\kappa}
\newcommand{\setm}{\setminus}
\newcommand{\DCC}{\mathcal{D}^\CC}
\newcommand{\pip}{\varphi}

\newcommand{\p}{{$p\mspace{1mu}$}}
\newcommand{\Np}{N^{1,p}}
\newcommand{\distin}{\dist_{\textup{in}}}
\newcommand{\bdry}{\partial}
\DeclareMathOperator{\capp}{cap}
\newcommand{\cp}{\capp_p}

\newcommand{\eqv}{\mathchoice{\quad \Longleftrightarrow \quad}
    {\Leftrightarrow}{\Leftrightarrow}{\Leftrightarrow}}
\newcommand{\Cp}{{C_p}}
\newcommand{\imp}{\mathchoice{\quad \Longrightarrow \quad}{\Rightarrow}
                {\Rightarrow}{\Rightarrow}}

\newcounter{saveenumi}

\begin{document}

\title{Dichotomy of global capacity density in metric measure spaces}

\author{Hiroaki Aikawa}
\address{
Department of Mathematics,
Hokkaido University,
Sapporo 060-0810, Japan
}
\email{aik@math.sci.hokudai.ac.jp}

\author{Anders Bj\"orn}
\address{Department of Mathematics,
Link\"oping University,
SE-581 83 Link\"oping, Sweden}
\email{anders.bjorn@liu.se}

\author{Jana Bj\"orn}
\address{Department of Mathematics,
Link\"oping University,
SE-581 83 Link\"oping, Sweden}
\email{jana.bjorn@liu.se}

\author{Nageswari Shanmugalingam}
\address{
Department of Mathematical Sciences,
P.O. Box 210025,
University of Cincinnati,
Cincinnati, OH 45221-0025, U.S.A.
}
\email{shanmun@uc.edu}


\thanks{The first 
author was partially supported  by 
JSPS KAKENHI grants 25287015 and 25610017.
The second and third authors were partially supported by the  
Swedish Research Council.
The last author was partially supported by NSF grants DMS-1200915
and DMS-1500440. Part of this research was conducted during the last author's visit to
Link\"oping University; she wishes to thank that institution for its kind hospitality.}

\begin{abstract}
The variational capacity $\capap$ in Euclidean spaces is known to enjoy the
density dichotomy at large scales, namely that for every $E\subset\R^n$,
\[
\inf_{x\in\R^n}\frac{\capap(E\cap B(x,r),B(x,2r))}{\capap(B(x,r),B(x,2r))}
\]
is either zero or tends to $1$ as $r \to \infty$.
We prove that this property still holds in unbounded
complete geodesic metric spaces equipped
with a doubling  measure supporting a \p-Poincar\'e inequality,
but that it can fail in nongeodesic metric spaces and also for the Sobolev
capacity in~$\R^n$.

It turns out that
the shape of balls impacts the validity of the density dichotomy. 
Even in more general metric spaces, we construct families of sets, such as
John domains, for which the density dichotomy holds.
Our arguments include an exact formula for the variational capacity of 
superlevel sets for capacitary potentials
and a quantitative approximation from inside of 
the variational capacity.
\end{abstract}

\maketitle

\section{Introduction}

In extending a result of Hayman and Pommerenke~\cite{HaPo-BMOA} and
giving a characterization of analytic functions mapping the unit disk into a given planar domain $\Om$,
Stegenga~\cite{Steg} came across a dichotomy property of the logarithmic capacity, namely that if $E\subset\R^2$
is the complement of a planar domain, then its logarithmic capacity density with respect to a radius $r>0$ either
tends to $0$ or to $1$ as $r\to\infty$. The property that the complement of $\Om$ has 
its logarithmic capacity density tending to
$1$ at global scales 
characterizes the property that analytic functions from the unit disk to $\Om$ belong to the class $\BMOA$.

In  \cite{MR3411153} the first author, together with Itoh,
 studied such a dichotomy property of the
global capacity density 
for the variational \p-capacity, $1<p<\infty$, in 
weighted Euclidean spaces.
In this note we investigate the same problem in the
nonsmooth setting of metric measure spaces, where it
is considerably more complicated and subtle.
It turns out that the dichotomy fails in general, and that
the shape of balls plays a significant role.

Fix $1<p<\infty$ and let 
$(X,d,\mu)$ be an unbounded complete metric measure space 
with  a doubling measure $\mu$ supporting a \p-Poincar\'e inequality.
It is known that such a metric space is $L$-quasiconvex for some $L\ge1$, i.e., 
for all $x,y\in X$, there exists a rectifiable curve $\gamma$
connecting $x$ and $y$ with length 
$\ell(\gamma) \le L d(x,y)$.
(See Section~\ref{sect-prelim} for this and other
facts mentioned in this introduction.) 
Define the \emph{inner metric} $\din$ by
\begin{equation} \label{eq:din}
\din(x,y)=\inf_\gamma \ell(\gamma),
\end{equation}
where the infimum is taken over all rectifiable curves $\gamma$ 
connecting $x$ and $y$.
It follows from the $L$-quasiconvexity that
$d(x,y)\le \din(x,y)\le Ld(x,y)$.
Moreover, arc length with respect to the given distance
$d$ and with respect to the inner metric $\din$ are the
same, and
thus $X$ is a geodesic space (i.e., $1$-quasiconvex) 
with respect to $\din$.

Now let $E\subset X$ and $\dconst>1$.
We study the following global lower capacity densities
\begin{align*}
\dens{r}{\dconst}{E}
&=\inf_{x\in X}
\frac{\capapED{E\cap B(x,r)}{B(x,\dconst r)}}
{\capapED{B(x,r)}{B(x,\dconst r)}},   \nonumber\\
\densin{r}{\dconst}{E}
&=\inf_{x\in X}
\frac{\capapED{E\cap\Bin(x,r)}{\Bin(x,\dconst r)}}{\capapED{\Bin(x,r)}
{\Bin(x,\dconst r)}}.  
\end{align*}
Here $B(x,r)=\{y\in X:d(x,y)<r\}$ and $\Bin(x,r)=\{y\in X:\din(x,y)<r\}$ 
denote the ordinary and inner balls, respectively,
and $\cp$ is the variational capacity (see \eqref{eq:varcap}).

It is easy to see that, as $r\to\infty$,
the limit of $\dens{r}{\dconst}{E}$ and that of $\densin{r}{\dconst}{E}$ are comparable
(see \reflem{dens=densin}).
However, they have different nature.
We show that $\densin{r}{\dconst}{E}$ has the same dichotomy as in the Euclidean case
found in~\cite[Corollary~1.5]{MR3411153},
whereas $\dens{r}{\dconst}{E}$ does not have such a dichotomy in general.
More precisely, we have the following two theorems.

\begin{thm}\label{thm:main}
For every $E\subset X$ one of the following statements holds:
\begin{enumerate}
\renewcommand{\theenumi}{\textup{(\roman{enumi})}}%
\renewcommand{\labelenumi}{\theenumi}%
\item 
\label{ref-i}
$\lim_{R\to\infty}\densin{R}{\dconst}{E}=0$, 

\item $\lim_{R\to\infty}\densin{R}{\dconst}{E}=1$.
\end{enumerate}
Furthermore, the two possibilities listed above are independent of 
$\dconst>1$, and
\ref{ref-i} holds if and only if any of the following
equivalent conditions is satisfied:        
\begin{enumerate}
\renewcommand{\theenumi}{\textup{(\alph{enumi})}}%
\renewcommand{\labelenumi}{\theenumi}%
\item 
$\lim_{R\to\infty}\dens{R}{\dconst}{E}=0$, 
\item 
         $\dens{r}{\dconst}{E}=0$ for all $r>0$,
\item 
         $\densin{r}{\dconst}{E}=0$ for all $r>0$.
\end{enumerate}
\end{thm}

\begin{thm}\label{thm:nodicho}
There exists a  complete
unbounded metric measure space $(X,d,\mu)$
supporting a $1$-Poincar\'e inequality with $\mu$ doubling
and $E\subset X$  such that
\[
0<\liminf_{R\to\infty}\dens{R}{\dconst}{E}<1
\quad \text{for all } \tau >1.
\]
\end{thm}

The above counterexample to the dichotomy arises from the lack of geodesics
with respect to the ordinary metric.
Although by the quasiconvexity of $X$, an ordinary ball $B(x,r)$ 
and an inner ball $\Bin(x,r)$ satisfy
\begin{equation}\label{eq:BinB}
\Bin(x,r)\subset B(x,r)\subset \Bin(x,Lr),
\end{equation}
and thus are comparable,
the ordinary balls may be oddly shaped.
This illustrates the difference between 
Theorems~\ref{thm:main} and~\ref{thm:nodicho}.
As was observed in \cite{MR3411153}, 
uniform approximation of capacity from inside
plays an important role for the dichotomy of the global capacity density.
Such an approximation property can be verified for
domains satisfying an interior corkscrew condition, see
\refsec{ACI} for details.
To further understand this phenomenon  we introduce the notion of capacitarily
stable collections of sets in Section~\ref{sect-cap-stable} and show
that the dichotomy holds for such collections. We also give examples
of capacitarily stable collections, including one consisting of John domains.

Even though there is no dichotomy of the type above for $\dens{R}{\dconst}{E}$,
we have the following weak dichotomy.

\begin{thm}\label{thm:main-weak}
Let $\dconst >1$. 
Then there is a constant $A>0$, depending only on $\dconst$, $p$ and $X$, 
such that for every $E\subset X$ one of the following statements holds:
\begin{enumerate}
\item 
$\lim_{R\to\infty}\dens{R}{\dconst}{E}=0$, 
\item 
$\liminf_{R\to\infty}\dens{R}{\dconst}{E}\ge A$.
\end{enumerate}
Furthermore, the two possibilities listed above are independent of $\dconst>1$,
with the exception that the constant $A$ depends on $\dconst$.
\end{thm}

One may ask if there can be a similar dichotomy
for other capacities as well.
In \cite{Aik-Riesz} the first author observed that the Riesz capacity of order $\alpha$
($0<\alpha\le2$) in the Euclidean space has the same dichotomy property.
On the other hand, we show in \refexa{Cp} that the Sobolev capacity $C_p$ has neither dichotomy nor weak dichotomy
even in the linear case $p=2$ on unweighted $\R^n$.
It would be interesting to characterize capacities whose global densities have dichotomy.

The outline of the paper is as follows.
In Section~\ref{sect-prelim} we introduce the necessary
background from nonlinear analysis
 on metric spaces.
In Section~\ref{sect-lower-cap-density}
 we recall some basic estimates for the variational capacity and 
use them to deduce comparison results for the capacity 
density functions $\densname$ and $\densinname$.
In Section~\ref{sect-superlevel} we deduce an
identity for the capacity of superlevel sets for the capacitary potentials. 
Similar estimates
have earlier been obtained in \cite{MR1869615}, but here we obtain
an exact identity.

In the subsequent
two sections, we give the proof of \refthm{main}, through
the use of a number of simpler 
lemmas. 
Also \refthm{main-weak} is obtained therein.
In Section~\ref{sect-counterex} we give the key counterexample yielding 
\refthm{nodicho}, and another counterexample  showing that there is
no dichotomy for the Sobolev capacity.

Finally, in the last section we define capacitarily stable collections,
show that they satisfy a dichotomy, and 
give examples of such families, including families of John domains and families
of domains satisfying the interior corkscrew condition.

\section{Notation and preliminaries}
\label{sect-prelim}

We assume throughout the paper
that $1<p<\infty$ and that $X=(X,d,\mu)$ is an unbounded
complete metric space equipped
with a metric $d$ and a doubling  measure $\mu$, i.e.,
there exists $C>0$ such that for all balls
$B=B(x_0,r):=\{x\in X: d(x,x_0)<r\}$ in~$X$,
\begin{equation*}
        0 < \mu(2B) \le C \mu(B) < \infty.
\end{equation*}
Here and elsewhere we let $\lambda B=B(x_0,\lambda r)$.
We will also assume that $X$ supports a \p-Poincar\'e inequality, see below,
and that $\Dom\subset X$ is a nonempty bounded open set.

Proofs of the results in 
this section, as well as historical comments, can be found in the monographs
\cite{MR2867756}  and
\cite{MR3363168}.

We will only consider curves which are nonconstant, compact
and 
rectifiable (i.e., have finite length), and thus each curve can 
be parameterized by its arc length $ds$. 
A property is said to hold for \emph{\p-almost every curve}
if it fails only for a curve family $\Ga$ with zero \p-modulus, 
i.e., there exists $0\le\rho\in L^p(X)$ such that 
$\int_\ga \rho\,ds=\infty$ for every curve $\ga\in\Ga$.

Following \cite{KoMc} and
\cite{HeKo98} we introduce weak upper gradients as follows.

\begin{deff} 
A measurable function $g : X \to [0,\infty]$  is a \emph{\p-weak upper gradient} 
of a function $f: X \to [-\infty,\infty]$
if for \p-almost every  
curve $\gamma : [0,\ell(\gamma)] \to X$,
\[
        |f(\gamma(0)) - f(\gamma(\ell(\gamma)))| \le \int_{\gamma} g\,ds,
\]
where the left-hand side is considered to be $\infty$ 
whenever at least one of the terms therein is infinite.
\end{deff}

If $f$ has a \p-weak upper gradient in $L^p(X)$, then
it has an a.e.\ unique \emph{minimal \p-weak upper gradient} $g_f \in L^p(X)$
in the sense that for every \p-weak upper gradient $g \in L^p(X)$ of $f$ we have
$g_f \le g$ a.e., see \cite{Sh-harm}.
Following \cite{Sh-rev}, 
we define a version of Sobolev spaces on the metric space $X$.

\begin{deff} \label{deff-Np}
For a measurable function $f: X\to[-\infty,\infty]$, let 
\[
        \|f\|_{\Np(X)} = \biggl( \int_X |f|^p \, d\mu 
                + \inf_g  \int_X g^p \, d\mu \biggr)^{1/p},
\]
where the infimum is taken over all \p-weak upper gradients $g$ of $f$.
The \emph{Newtonian space} on $X$ is 
\[
        \Np (X) = \{f: \|f\|_{\Np(X)} <\infty \}.
\]
\end{deff}
\medskip

The space $\Np(X)/{\sim}$, where  $f \sim h$ if and only if $\|f-h\|_{\Np(X)}=0$,
is a Banach space,  see \cite{Sh-rev}.
In this paper we assume that functions in $\Np(X)$
 are defined everywhere (with values in $[-\infty,\infty]$),
not just up to an equivalence class in the corresponding function space.
Note that a modification of an $\Np(X)$-function on a set of measure
zero does not necessarily belong to $\Np(X)$.

The (Sobolev) \emph{capacity} of an arbitrary set $E\subset X$ is
\[
\Cp(E) = \inf_u\|u\|_{\Np(X)}^p,
\]
where the infimum is taken over all $u \in \Np(X)$ such that
$u\geq 1$ on $E$.
A property holds \emph{quasieverywhere} (q.e.)\ 
if the set of points  for which the property does not hold 
has capacity zero. 
The capacity is the correct gauge 
for distinguishing between two Newtonian functions. 
If $u \in \Np(X)$, then $u \sim v$ if and only if $u=v$ q.e.
Moreover, 
if $u,v \in \Np(X)$ and $u= v$ a.e., then $u=v$ q.e.

\begin{deff} \label{def.PI.}
We say that $X$ supports a \emph{\p-Poincar\'e inequality}  if
there exist constants $C>0$ and $\lambda \ge 1$
such that for all balls $B \subset X$,
all integrable functions $f$ on $X$ and all \p-weak upper gradients $g$ of $f$,
\[
        \vint_{B} |f-f_B| \,d\mu
        \le C \diam(B) \biggl( \vint_{\lambda B} g^{p} \,d\mu \biggr)^{1/p},
\]where $ f_B
 :=\vint_B f \,d\mu
:= \int_B f\, d\mu/\mu(B)$.
\end{deff}

\emph{From now on we assume that $X$ supports a \p-Poincar\'e inequality.}

\medskip

Let $\Om\subset X$ be open.
We define the \emph{variational capacity}
 $\capapED{E}{\Dom}$ of $E \subset \Om$ by
\begin{equation} \label{eq:varcap}
\capapED{E}{\Dom}
=\inf_u\int_\Dom g_u^p\,d\mu,
\end{equation}
where the infimum is taken over all $u\in\Np(X)$ such that $u=1$ q.e.\ on $E$
and $u=0$ everywhere on $X \setm \Om$; we call such functions \emph{admissible}.
(One can equivalently assume that $u=1$ (quasi)everywhere  on $E$ 
and $u=0$ (quasi)everywhere on $X \setm \Om$.)

If there is an admissible function $u$ (which happens if and only if 
$\capapED{E}{\Dom}< \infty$), then there is also a minimizer of the problem 
\eqref{eq:varcap} and 
it is unique up to sets of capacity zero. Moreover, there
is a unique minimizer 
$u_E^\Dom$ which is also  \emph{lower semicontinuously regularized} in $\Om$, i.e.,
\[
 u_E^\Dom(x):=\essliminf_{y\to x} u_E^\Dom(y)
   := \lim_{r \to 0}\, \Bigl(\essinf_{B(x,r)} u_E^\Dom\Bigr),
 \quad x \in \Om.
\]
This unique minimizer $u_E^\Dom$ is 
the \emph{capacitary potential} of $E$ in $\Dom$;
it is also referred to as the capacitary potential 
for $\capapED{E}{\Dom}$.
When it exists,
the capacitary potential $u_E^\Dom$
satisfies
\begin{equation}\label{eq:gEDp}
\capapED{E}{\Dom}=\int_\Dom g_{E}^p\,d\mu,
\end{equation}
where $g_E$ is the  minimal \p-weak upper gradient of $u_E^\Dom$.
By definition $u_E^\Dom(x)=0$ for $x \notin \Om$.

Under our assumptions, $(X,d)$ is $L$-quasiconvex,
with $L$ depending only on $p$ and $X$.
Here and below, when we say that a constant depends on 
$p$ and $X$ we really mean that it depends on $p$,
the doubling constant  and the constants in the \p-Poincar\'e 
inequality.
It follows from the quasiconvexity that
the inner metric 
(as defined in \eqref{eq:din})
is indeed a metric on $X$.
Moreover, arc length for curves is the same with respect to $d$ and $\din$.
Thus the class of \p-weak upper gradients of a function is
also the same
with respect to both metrics, and as a consequence
$\Np(X)$ is the same for both metrics.
Moreover, $(X,\din,\mu)$ satisfies our doubling and Poincar\'e
assumptions,
 and thus
the theory is directly applicable also with respect to $\din$.

We say that two nonnegative quantities $a$ and $b$ are \emph{comparable},
and write $a\simeq b$, if  
$a/C \le b \le Ca$
for
some constant $\const\ge1$, where the
constant $\const$ is referred to as the \emph{constant of comparison}.

\section{Comparison of global lower capacity densities}
\label{sect-lower-cap-density}

We recall some  well-known estimates for the capacity in balls.

\begin{lem}[{\cite[Proposition~6.16 and Lemma~11.22]{MR2867756}}]
\label{lem:capball}
Let $0<a<b$.
Then 
\begin{equation} \label{eq:cap-est-0}
\capapED{B(x,ar)}{B(x,br)}\simeq 
r^{-p}\mu(B(x,r)), 
\end{equation}
where the constant of comparison depends
only on $a$, $b$, $p$ and $X$.
Moreover, if $1<s<t$, then
\begin{equation} \label{eq:cap-est}
\capapED{E}{B(x,tr)}\le\capapED{E}{B(x,sr)}\le\const\capapED{E}{B(x,tr)}
\qtext{for $E\subset B(x,r)$,}\\
\end{equation}
where $\const>1$ depends only on $s$, $t$, $p$ and $X$. 

The corresponding estimates with
respect to the inner metric also hold.
\end{lem}

Using the estimates above,
we can
show that $\dens{r}{\dconst}{E}$ and $\densin{r}{\dconst}{E}$ are comparable
in the following sense.

\begin{lem}\label{lem:dens=densin}
Let $\tau,\tau'>1$.
For every $r>0$ and $E\subset X$ we have
\begin{equation} \label{eq:D-Din-1}
\densin{r}{\dconst}{E}
\le\const\dens{r}{\dconst}{E}
\le\const^2\densin{Lr}{\dconst}{E},
\end{equation}
where $\const>1$ depends only on 
$\tau$, $p$ and $X$, and $L$ is the quasiconvexity constant.
Moreover,
\[ 
\dens{r}{\dconst}{E}\simeq \dens{r}{\dconst'}{E}
\qtext{and}\quad
\densin{r}{\dconst}{E}\simeq \densin{r}{\dconst'}{E},
\] 
where the constants of comparison depend only on 
$\tau$, $\tau'$, $p$ and $X$.
\end{lem}

\begin{proof}
In view of \refeq{BinB} and \reflem{capball} we see that
\begin{align*}
\capap(E\cap\Bin(x,r),\Bin(x,\tau r))
& \simeq \capap(E\cap\Bin(x,r),\Bin(x,\tau Lr))\\
&\le \capap(E\cap\Bin(x,r),B(x,\tau r)) \\
&\le \capap(E\cap B(x,r),B(x,\tau r))
\end{align*}
and
\begin{align*}
\capap(E\cap B(x,r),B(x,\tau r))
&\simeq
\capap(E\cap B(x,r),B(x,\tau L r))\\
&\le\capap(E\cap \Bin(x,Lr),B(x,\tau L r)) \\
&\le\capap(E\cap \Bin(x,Lr),\Bin(x,\tau L r)),
\end{align*}
with constants of comparison depending only on 
$\tau$, $p$ and $X$.
Hence (using \refeq{BinB} and \refeq{cap-est-0}
to see that the denominators are
comparable), 
\begin{align*}
\frac{\capap(E\cap\Bin(x,r),\Bin(x,\tau r))}{\capap(\Bin(x,r),\Bin(x,\tau r))}
&\le
\const'\frac{\capap(E\cap B(x,r),B(x,\tau r))}{\capap(B(x,r),B(x,\tau r))}\\
&\le
\const''\frac{\capap(E\cap\Bin(x,Lr),\Bin(x,\tau Lr))}{\capap(\Bin(x,Lr),\Bin(x,\tau Lr))}.
\end{align*}
Taking the infima with respect to $x\in X$ 
yields \eqref{eq:D-Din-1}.
The last assertion follows directly from
\eqref{eq:cap-est} (and the corresponding estimate in the inner metric).
\end{proof}

\section{Capacity of superlevel sets of a capacitary potential}
\label{sect-superlevel}

In this section we evaluate
the capacity
of superlevel sets 
of the capacitary potential, which may be of independent interest.

\begin{prop}\label{prop:capapEM}
Let $E\subset\Dom$ with
$\cp(E,\Dom)<\infty$, and
$u_E$ be the capacitary potential  of $E$ in $\Dom$.
For $0<M \le 1$ let $E_M=\{x\in\Dom:u_E(x)>M\}$ and
$E'_M=\{x\in\Dom:u_E(x)\ge M\}$.
Then
\begin{alignat*}{2}
\capapED{E_M}{\Dom}  
&=M^{1-p}\capapED{E}{\Dom}, &\quad& \text{if } 0 < M < 1, \\
\capapED{E'_M}{\Dom}
&=M^{1-p}\capapED{E}{\Dom}, &\quad& \text{if } 0 < M \le 1. 
\end{alignat*}
\end{prop}

This result was obtained for weighted $\R^n$ (with a \p-admissible weight)
in \cite[p.\ 118]{MR1207810}.
Their argument depends on the Euler--Lagrange equation,
which is not available in the metric space setting considered here.
Nevertheless, the weaker estimate 
\begin{equation}    \label{eq-weaker-est-E-M}
\capapED{E_M}{\Dom}\simeq M^{1-p}\capapED{E}{\Dom}
\end{equation}
was obtained in \cite[Lemma~5.4]{MR1869615} via a variational approach.
Our proof of
\refprop{capapEM} is also based on the variational method, 
yet it yields the sharp identity in the metric space setting and
is shorter
than the earlier proofs of \eqref{eq-weaker-est-E-M}
and the proof in \cite[pp.\ 116--118]{MR1207810}.

\begin{proof}
For simplicity, write $g_E$ for the minimal \p-weak upper gradient of $u_E$.
It follows from \cite[Lemma~11.19]{MR2867756} that
\begin{equation} \label{eq-EM}
\capapED{E_M}{\Dom}
=\capapED{E'_M}{\Dom}
=\frac{1}{M^p} \int_{0<u_E<M} g_{E}^p\,d\mu,
\quad \text{if } 0 <M < 1.
\end{equation}
The second equality in \eqref{eq-EM} also holds when $M=1$ (and 
is easier to deduce than for $M<1$).
 Hence, by \refeq{gEDp}, it suffices to show that
\begin{equation}\label{eq:uE<M}
\int_{0<u_E<M} g_{E}^p\,d\mu=M\int_{0<u_E<1} g_{E}^p\,d\mu.
\end{equation}
For $0<t<1$ define the piecewise linear function $\Phi_t(s)$ on $[0,\infty)$ by
\[
\Phi_t(s)=
\begin{cases}
\dfrac {ts}{M} &\text{for }0\le s<M,\\
t+\dfrac{1-t}{1-M}(s-M)&\text{for }M\le s<1,\\
1&\text{for } s\ge1.
\end{cases}
\]
We note that $g_E$ vanishes a.e.\ on each level set 
$\{x\in \Om: u_E(x)=t\}$. Therefore, for
each  $0<t<1$ we see that $v_t(x):=\Phi_t(u_E(x))$ 
is admissible for $\capapED{E}{\Dom}$ and
\[
\phi(t):=
\int_{\Om} g_{v_t}^p\,d\mu
=\biggl(\frac tM\biggr)^p\int_{0<u_E<M} g_E^p\,d\mu
 +\biggl(\frac{1-t}{1-M}\biggr)^p \int_{M<u_E<1} g_E^p\,d\mu.
\]
By definition, $\Phi_M(s)=s$ for $0\le s\le 1$, and so $v_M(x)=u_E(x)$.
Hence
\[
\varphi(t)\ge\varphi(M)=\int_{0<u_E<1}g_{E}^p\,d\mu=\capapED{E}{\Dom}
\]
with equality for $t=M$.
In particular $\varphi'(M)=0$.
Since
\[
\varphi'(t)
=\frac{pt^{p-1}}{M^p}\int_{0<u_E<M}g_{E}^p\,d\mu
-\frac{p(1-t)^{p-1}}{(1-M)^p}\int_{M<u_E<1}g_{E}^p\,d\mu,
\]
it follows from $\varphi'(M)=0$ that
\[
\frac1M\int_{0<u_E<M} g_{E}^p\,d\mu
=\frac1{1-M}\int_{M<u_E<1} g_{E}^p\,d\mu
=\frac1{1-M}\biggl(\int_{0<u_E<1} g_{E}^p\,d\mu-\int_{0<u_E<M} g_{E}^p\,d\mu\biggr),
\]
which yields \refeq{uE<M}.
\end{proof}

\section{Lower estimate of capacity density}

We now use \refprop{capapEM} to deduce estimates for 
the ratio of capacities
in terms of the infimum of the corresponding capacitary potential.

\begin{lem}[cf.\ {\cite[Lemma~4.2]{MR3411153}}]\label{lem:uE<}
Let $A,E\subset\Dom$ 
with 
$\capapED{A}{\Dom}>0$
and $\capapED{E}{\Dom}<\infty$.
Then the  capacitary potential $u_E$ of $E$ in $\Dom$
satisfies
\begin{equation}    \label{eq-inf-le-cap}
\inf_A u_E\le\biggl(\frac{\capapED{E}{\Dom}}{\capapED{A}{\Dom}}\biggr)^{1/(p-1)}.
\end{equation}
\end{lem}

\begin{proof}
If 
 $\inf_A u_E=0$, then \eqref{eq-inf-le-cap}
holds trivially. Now suppose that
$M=\inf_A u_E>0$. 
Then $A\subset E'_M:=\{x\in\Dom: u_E(x)\ge M\}$.
\refprop{capapEM} yields
\[
\capapED{A}{\Dom}
\le\capapED{E'_M}{\Dom}
=M^{1-p}\capapED{E}{\Dom},
\]
which readily gives the required inequality.
\end{proof}

When $A$ is a ball,
there is a converse 
inequality 
to~\eqref{eq-inf-le-cap}
up to a multiplicative constant
depending only on $p$ and $X$.
Let $\Lambda=100\lambda$ with $\lambda\ge1$ being 
the dilation constant in the \p-Poincar\'e inequality.

\begin{lem}[{\cite[Lemma~11.20]{MR2867756}}]\label{lem:uE>}
There exists 
$0<C_0\le1$, depending only on $p$ and $X$, such that if
$E\subset\itoverline{B(x,r)}$, then
the capacitary potential $u_E$ of 
$E$ in $B(x,\Lambda r)$ satisfies
\[
\inf_{B(x,2r)}u_E
\ge C_0
\biggl(\frac{\capapED{E}{B(x,\Lambda r)}}{\capapED{B(x,2 r)}{B(x,\Lambda r)}}\biggr)^{1/(p-1)}.
\]
\end{lem}

In metric spaces such an estimate was obtained in
\cite[Lemma~5.6]{MR1869615}
and \cite[Lemma~3.9]{JB-pfine}
under additional
assumptions.
For the reader's convenience, we sketch how this can 
be proved using
Proposition~\ref{prop:capapEM}.

\begin{proof}[Sketch of proof]
Let $M=\sup_{\bd B(x,2r)}u_E$ and let $E_M=\{x\in B(x,\Lambda r):u_E(x)> M\}$.
Then by the minimum principle 
(see \cite[Proposition~9.4 and Theorem~9.13]{MR2867756}),
we get that $E_M\subset B(x,2r)$.
Hence by \refprop{capapEM},
\[
M^{1-p}\capapED{E}{B(x,\Lambda r)}
=\capapED{E_M}{B(x,\Lambda r)}
\le\capapED{B(x,2r)}{B(x,\Lambda r)}
\]
so that
\[ 
M\ge\biggl(
\frac{\capapED{E}{B(x,\Lambda r)}}{\capapED{B(x,2r)}{B(x,\Lambda r)}}
\biggr)^{1/(p-1)}.
\] 
Finally using weak Harnack inequalities it can be shown that
$
\inf_{B(x,2r)}u_E \ge C_0 M,
$
see the proof in \cite{MR2867756}.
\end{proof}

The following lemma is a variant of a comparison principle for capacitary
potentials and will be useful when proving the subsequent results.

\begin{lem}\label{lem:comp-princ-cap-pot}
Let $V\subset\Dom$ be open 
and let $E'\subset E\subset\Dom$ be arbitrary
sets such that $E'\subset V$, $\cp(E',V)<\infty$ and $\cp(E,\Om)<\infty$.
Let $u^V_{E'}$ and $u^\Dom_E$ be the corresponding capacitary potentials
and assume that $1-u^\Dom_E\le a$ on $\bdry V$
with $0 \le a \le 1$.
Then $1-u^\Dom_E\le a (1-u^V_{E'})$ in~$V$.
\end{lem}

\begin{proof}
By the minimum principle,
$u^\Dom_E\ge 1-a$ in $V$.
Hence
\[
v:= (u^\Dom_E-1)+a\ge0 \quad \text{in } V,
\]
and it is easily verified that $v$ is the lower semicontinuously
regularized 
solution of the obstacle problem
(see \cite[Definition~7.1]{MR2867756})
in $V$ with the obstacle $a\chi_E$ and boundary data $v\ge0$
on $\bdry V$.
Applying the comparison principle (\cite[Lemma~8.30]{MR2867756}) to $v$
and $au^V_{E'}$ 
shows that
\[
(u^\Dom_E-1)+a = v \ge a u^V_{E'} \qtext{in } V,
\]
from which the lemma follows.
\end{proof}

For an open set $U$
we let $\disU(x)=\dist(x,X \setm U)$
and define the \emph{$\ve$-interior} $U_\ve$  of $U$ 
by 
\begin{equation} \label{eq:U_eps}
U_\ve=\{x\in U: \disU(x)> \ve\}.
\end{equation}
Iterating \reflem{uE>}, we obtain the following lemma.

\begin{lem}\label{lem:hmp<onUk}
Let $U\subset\Dom$ be open, $0<\eta<1$,
and $r>0$.
Suppose that $E\subset U$ satisfies
\begin{equation}\label{eq:C/C>eta}
\frac{\capapED{E\cap B(x,r)}{B(x,\Lambda r)}}{\capapED{B(x,2r)}{B(x,\Lambda r)}}
\ge \eta
\qtext{for  
every }x\in U_{\Lambda r}.
\end{equation}
If $k$ is a positive integer and $U_{k\Lambda r}\ne\emptyset$, 
then
\begin{equation}\label{eq:hmp<()k}
1-u_E^\Dom\le (1-C_0\eta^{1/(p-1)})^k
\qtext{in $U_{k\Lambda r}$,}
\end{equation}
where $0 < C_0 \le 1$ is as in Lemma~\ref{lem:uE>}.
\end{lem}

\begin{proof}
Since
$0< C_0 \le 1$ we see that
$0<1-C_0\eta^{1/(p-1)}<1$. 
Take an arbitrary point $x\in 
U_{\Lambda r}$ and let $B=B(x,r)$.
By \reflem{uE>} with $E\cap B$ in place of $E$ and by \refeq{C/C>eta} we 
see that the capacitary potential $u_{E\cap B}^{\La B}$ of $E\cap B$ in $\La B$
satisfies
\begin{equation}   \label{eq-u-La-B}
1-u_{E\cap B}^{\La B}\le 1-C_0\eta^{1/(p-1)} \quad \text{in } B.
\end{equation}
We prove \refeq{hmp<()k} by induction on $k$
using \eqref{eq-u-La-B}.
Since $\disU(x)>\Lambda r$, we see that 
$\La B\subset U\subset\Dom$, and hence
by Lemma~\ref{lem:comp-princ-cap-pot} and \eqref{eq-u-La-B},
\[
1-u_E^\Dom \le 1-u_{E\cap B}^{\La B} \le 1-C_0\eta^{1/(p-1)}
\qtext{in $B$}.
\]
Since $x\in U_{\Lambda r}$ was arbitrary,
we obtain \eqref{eq:hmp<()k} for $k=1$.

Now let $k\ge2$ and assume that \refeq{hmp<()k} holds with $k-1$
in place of $k$.
Let $x\in U_{k\Lambda r}$ be arbitrary.
Another application of Lemma~\ref{lem:comp-princ-cap-pot} (with $V=\La B$
and $a=(1-C_0\eta^{1/(p-1)})^{k-1}$),
together with~\eqref{eq-u-La-B}, shows that
\[
1-u^\Om_E \le (1-C_0\eta^{1/(p-1)})^{k-1} (1-u_{E\cap B}^{\La B})
\le (1-C_0\eta^{1/(p-1)})^k
\qtext{in } B,
\]
which, due to the arbitrariness of $x\in U_{k\Lambda r}$, amounts 
to~\eqref{eq:hmp<()k}.
This completes the induction.
\end{proof}

Lemmas~\ref{lem:hmp<onUk} and~\ref{lem:uE<} (the latter with 
$U_{k\Lambda r}$ and $E\cap U$ in place of $A$ and $E$)
readily give the following lower bound for the ratio of capacities.
(For $x\in U_{\La r}$ we have $E\cap U\cap B(x,r)=E\cap B(x,r)$
so that \refeq{C/C>eta} holds with $E\cap U$ 
in place of $E$ for such $x$.)

\begin{cor}\label{cor:C/C>}
Let $U\subset\Dom$ be open,  $0<\eta<1$
and $r>0$.
Suppose that $E\subset X$ satisfies \refeq{C/C>eta}.
If $k$ is a positive integer and $U_{k\Lambda r}\ne\emptyset$, then
\[
\frac{\capapED{E\cap U}{\Dom}}{\capapED{U_{k\Lambda r}}{\Dom}}
\ge (1-(1-C_0\eta^{1/(p-1)})^k)^{p-1}.
\]
\end{cor}

\begin{rem} 
Results analogous to those in this section for the inner metric 
follow immediately, as seen from
the discussion in the penultimate paragraph of Section~\ref{sect-prelim}.
\end{rem}

In the next section, we shall see that, if $R$ is large, then
$\capap(\Bin(x,R-k\Lambda r),\Bin(x,\tau R))$
is close to $\capap(\Bin(x,R),\Bin(x,\tau R))$ uniformly for $x\in X$.
This property does not hold for ordinary balls.
This is the reason why
$\densin{r}{\tau}{E}$ has dichotomy
and yet
$\dens{r}{\tau}{E}$ does not.

\section{Uniform approximation of capacity from inside 
and Proof of \refthm{main}}
\label{sec:ACI}

\emph{Let $U$ be an open set and recall from~\eqref{eq:U_eps} that 
$U_\ve=\{x\in U:\disU(x)>\ve\}$
is the $\ve$-interior  of~$U$.}
We also define the \emph{$\ve$-neighborhood} of $U$ by
\begin{equation} \label{eq:U[ve]}
U[\ve]=\{x\in\Rd:\dist(x,U)<\ve\}.
\end{equation}

\medskip

The main aim of this section is to prove
\refthm{main}. In order to do so we will
show that the capacity of $U_\eps$ approximates the capacity
of $U$, under suitable assumptions on~$U$.

\begin{deff} 
Let $0<\kappa<1$ and $0\le R_1<R_2$.
We say that $U$ satisfies the \emph{interior corkscrew condition}
with parameters $\kappa$, $R_1$ and $R_2$ if
\[ 
\text{$x\in U$ and $R_1< r< R_2$}
\quad \Longrightarrow \quad
\text{$U\cap B(x,r)$ contains a ball of radius $\kappa r$.}
\] 
\end{deff}

\begin{rem}   \label{rem-Bin-corkscrew}
For $R>0$, $\Bin(x,R)$ satisfies the interior corkscrew condition
with parameters $1/2L$, $0$ and $R$.
The same is not true in general for ordinary balls, cf.\
\refprop{nodichotomy}.
\end{rem}

\begin{lem}\label{lem:Ue}
Suppose that
$U$ satisfies the interior corkscrew condition
with parameters $\kappa$, $0$ and $R_2$.
Let $0<\ve<\kappa R_2/2$.
Then:
\begin{enumerate}
\renewcommand{\theenumi}{\textup{(\roman{enumi})}}%
\renewcommand{\labelenumi}{\theenumi}%
\item
\label{ref-a}
For every $x\in U$ and $2\ve/\kappa\le r<R_2$, the set $U_\ve\cap B(x,r)$
contains a ball of radius $\kappa r/2$. In particular,
$U_\ve$ satisfies the interior corkscrew condition with parameters
$\kappa/2$, $2\ve/\kappa$ and $R_2$.

\item 
\label{ref-b}
$U\subset U_\ve[2\ve/\kappa]$.
\end{enumerate}
\end{lem}

\begin{proof}
\ref{ref-a}
Let $x\in U$ and $2\ve/\kappa\le r<R_2$.
By hypothesis there is a ball $B(y,\kappa r)\subset U\cap B(x,r)$.
This means that $\disU(y)\ge\kappa r\ge2\ve$, so that
\[
\disof{U_\ve}(y)
\ge\disU(y)-\ve
\ge \kappa r -\frac{\kappa r}2
=\frac{\kappa r}2.
\]
Hence $B(y,\kappa r/2)\subset U_\ve\cap B(x,r)$.

\ref{ref-b}
Let $x\in U$ and apply (i) with $r=2\ve/\kappa$.
We find a ball $B(y,\ve)\subset U_\ve\cap B(x,2\ve/\kappa)$.
Then $y\in U_\ve$ and $d(x,y)<2\ve/\kappa$, so that
$x\in U_\ve[2\ve/\kappa]$.
\end{proof}

\begin{lem}\label{lem:U2-j}
Suppose that
$U \subset \Om$ satisfies the interior corkscrew condition
with parameters $\kappa$, $R_1$ and  
$R_2\le\dist(U,X \setm \Om)$.
If $j\ge1$ and $R_2/\Lambda^{j}>R_1$, then
\[
\frac{\capapED{U[R_2/\Lambda^{j}]}{\Dom}}{\capapED{U}{\Dom}}
\le(1-\eta_\kappa^j)^{1-p},
\]
where $0<\eta_\kappa<1$ depends only on $\kappa$, $p$ 
and  
$X$.
\end{lem}

\begin{proof}
Let $x\in U$ be arbitrary.
In view of \reflem{capball},
we find $0<\eta<1$ depending only on 
$\kappa$, $p$ and  
$X$ such that
if $x\in U$, then 
\[
\frac{\capapED{U\cap B(x,r)}{B(x,\Lambda r)}}
{\capapED{B(x,2r)}{B(x,\Lambda r)}}\ge \eta
\qtext{for all } R_1<r<R_2.
\]
This, together with 
\reflem{uE>}, yields for $R_1<r<R_2$,
\begin{equation}\label{eq:k0}
1-u_r \le 1-C_0\eta^{1/(p-1)} =: \eta_\kappa \qtext{in }B(x,r),
\end{equation}
where $u_r$ is the capacitary potential of $U\cap B(x,r)$ in
$B(x,\Lambda r)$ and 
$0<\eta_\kappa<1$ depends only on $\kappa$, $p$ and $X$.
Let $u_{U}$ be the capacitary potential of $U$ in $\Dom$.
We shall show that \eqref{eq:k0} implies
\begin{equation}\label{eq:uU-le-1-eta}
1-u_U \le \eta_\kappa^j
\quad \text{in } B(x,R_2/\Lambda^j)
\end{equation}
whenever $R_2/\Lambda^{j}>R_1$. 
To start with, note that \reflem{comp-princ-cap-pot}
(with $V=B(x,R_2/\La)$ and $a=1$)
and \refeq{k0} imply
\[
1-u_U \le 1-u_{R_2/\La} \le \eta_\kappa
\quad \text{in } B(x,R_2/\La),
\]
i.e., \refeq{uU-le-1-eta} holds for $j=1$.
Now let $j\ge 2$ and assume that
\refeq{uU-le-1-eta} holds with $j$ 
replaced by $j-1$.
As $R_2/\Lambda^{j}>R_1$, we know that 
\refeq{k0} holds for $r=R_2/\Lambda^{j}$.
Now, 
applying \reflem{comp-princ-cap-pot} with $V=B(x,R_2/\Lambda^{j-1})$
and $a=\eta_\kappa^{j-1}$,
yields
\[
1-u_U \le \eta_\kappa^{j-1} (1-u_{R_2/\Lambda^j}) \le \eta_\kappa^j
\quad \text{in } B(x,R_2/\Lambda^{j}),
\]
which proves \refeq{uU-le-1-eta} also for $j$.
Since $x\in U$ was arbitrary, we conclude that
\[
U[R_2/\Lambda^{j}]\subset\{x\in\Dom:u_U(x)\ge 1-\eta_\kappa^j\}.
\]
Hence \refprop{capapEM} yields the required inequality.
\end{proof}

By Lemmas~\ref{lem:Ue} and~\ref{lem:U2-j} with $U_\ve$ in place of $U$ 
we immediately obtain the following approximation of capacity from inside.

\begin{lem}\label{lem:approx.in}
Suppose that
$U \subset \Om$ satisfies the interior corkscrew condition 
with parameters $\kappa$, $0$ and $R_2\le \dist(U,X \setm \Om)$.
Let $0<\eta_{\kappa/2}<1$ be the constant in \reflem{U2-j} 
corresponding to $\kappa/2$.
If $j\ge1$ and $\eps \le \kappa R_2/2\Lambda^j  $,
then
\[
\frac{\capapED{U}{\Dom}}{\capapED{U_\ve}{\Dom}}
\le(1-\eta_{\kappa/2}^j)^{1-p}.
\]
\end{lem}

\begin{proof}[Proof of \refthm{main}]
In view of \reflem{dens=densin},
it is sufficient to show that if $\dens{r}{\La}{E}>0$ 
for some $r>0$,  
then $\lim_{R\to\infty}\densin{R}{\tau}{E}=1$.
Note that $\dens{r}{\La}{E}>0$ implies \refeq{C/C>eta} 
for some $0<\eta<1$.
Take an arbitrary positive number $\alpha<1$ and
find a positive integer $k$ such that
\[
(1-(1-C_0\eta^{1/(p-1)})^k)^{p-1}\ge\alpha,
\]
where $C_0$ is the constant
from 
Corollary~\ref{cor:C/C>}.
By Remark~\ref{rem-Bin-corkscrew}, 
$\Bin:=\Bin(x,R)$ satisfies the corkscrew condition
with parameters $\ka=1/2L$, $0$ and $R_2=\min\{1,\tau-1\}R$.
Corollary~\ref{cor:C/C>}, together with
Lemma~\ref{lem:approx.in} (and $U=\Bin$, $\Om=\tau\Bin=\Bin(x,\tau R)$
and $\eps=k\La r$), then implies that
\[
\densin{R}{\tau}{E} = \inf_{x \in X}
\frac{\cp(E\cap\Bin,\tau\Bin)}
{\cp((\Bin)_{k\La r},\tau\Bin)} \,
\frac{\cp((\Bin)_{k\La r},\tau\Bin)}{\capap(\Bin,\tau\Bin)}
\ge \alpha (1-\eta_{\kappa/2}^j)^{p-1},
\]
where 
$j$ is the maximal integer such that
\[
  k\La r\le \frac{\ka\min\{1,\tau-1\}R}{2\La^j}.
\]
Letting $R\to\infty$ (and thus $j \to \infty$)
and then $\alpha \to 1$ shows that
$\lim_{R\to\infty} \densin{R}{\tau}{E}=1$,
since clearly $\densin{R}{\tau}{E} \le 1$ for all $R>0$.
\end{proof}

\begin{proof}[Proof of \refthm{main-weak}]
This follows directly from \refthm{main} and \reflem{dens=densin}.
\end{proof}

\section{Counterexamples and proof of \refthm{nodicho}}
\label{sect-counterex}

In this section we shall first
construct an example $(X,d,\mu)$ for which
the dichotomy for ordinary balls does not hold.
Let $B^+(x,r)=\{y\in B(x,r):y^\di>x^\di\}$ with $x=(x^1,\dots,x^\di)\in \R^n$.
This is the open upper half ball in $\R^n$ with center at $x$ and radius $r$.
The half-open lower half ball is denoted by 
$B^-(x,r):=B(x,r) \setm B^+(x,r)$.

Let $x_j=(4^j,0,\dots,0)$ and $R_j=2^j$, $j=1,2,\ldots$.
Let $X=\Rd\sm \union_{j=1}^\infty B^+(x_j,R_j)$ and let $d(x,y)$
be the restriction of the Euclidean distance to $X$.
We write $\BX(x,r)=\{y\in X:d(x,y)<r\}$ for the open ball 
with center at $x$ and radius $r$ in $X$ with respect to $d(x,y)$.
Observe that $\BX(x,r)=B(x,r)\cap X$ with $B(x,r)$ being the Euclidean ball
with center at $x$ and radius $r$.
Let $\mu$ be the restriction of $\di$-dimensional Lebesgue measure on $X$.
Then $\mu$ is doubling on $X$.
Moreover, $X$ is the closure of a uniform domain in
$\R^n$ and hence supports a $1$-Poincar\'e
inequality, by \cite[Theorem~4.4]{BjShJMAA} and \cite[Proposition~7.1]{AS05}.
We will denote the variational capacities
with respect to $X$ and $\R^n$ by $\cp$
and $\capapRn$, respectively.

\begin{prop} \label{prop:nodichotomy}
Let  $1<p<\di$ and $\tau>1$.
In the situation described above
the following assertions hold true:
\begin{enumerate}
\renewcommand{\theenumi}{\textup{(\roman{enumi})}}%
\renewcommand{\labelenumi}{\theenumi}%
\item \label{k-unif-appr}
The balls $\BX(x,r)$
fail the 
uniform approximation of capacity.
More precisely, if $\rho>0$,
then $R_j/(R_j+\rho)\increase1$, as $j \to \infty$, and yet
for $2^j\ge 4\max\{\tau,\rho\}$,
\[
\frac
{\capapED{\BX(x_j,R_j)}{\BX(x_j,\tau(R_j+\rho))}}
{\capapED{\BX(x_j,R_j+\rho)}{\BX(x_j,\tau(R_j+\rho))}}
\le C <1,
\]
where $C$ is independent of $\rho$. 

\item \label{k-dich}
No dichotomy property, with respect to the balls $B_X(x,r)$, holds for
the set 
\begin{equation}\label{eq:E=BzdH}
E = \bigcup_{z\in\Z^\di \setm H} B(z,\de),
\quad \text{where } 0 < \de \le \tfrac{1}{4}, \quad
H=\bigcup_{j=1}^\infty B^+(x_j,R_j)\bigl[\tfrac12\bigr] 
\end{equation}
and $B^+(x_j,R_j)\bigl[\tfrac12\bigr]$ is the
  $\tfrac12$-neighborhood of $B^+(x_j,R_j)$,
here taken with respect to $\R^n$,
  see \eqref{eq:U[ve]} and Figure~\ref{fig:nodicho}.
More precisely,
\begin{enumerate}
\renewcommand{\theenumii}{\textup{(\alph{enumii})}}%
\renewcommand{\labelenumii}{\theenumii}%
\item  
$\dens{2\sqrt\di}{\tau}{E}>0$,
\item 
\(
0<\liminf_{R\to\infty} \dens{R}{\tau}{E}<1.
\)
\end{enumerate}
\end{enumerate}
\end{prop}

\begin{figure}[htb]
\footnotesize
\begin{overpic}[scale=.16]
{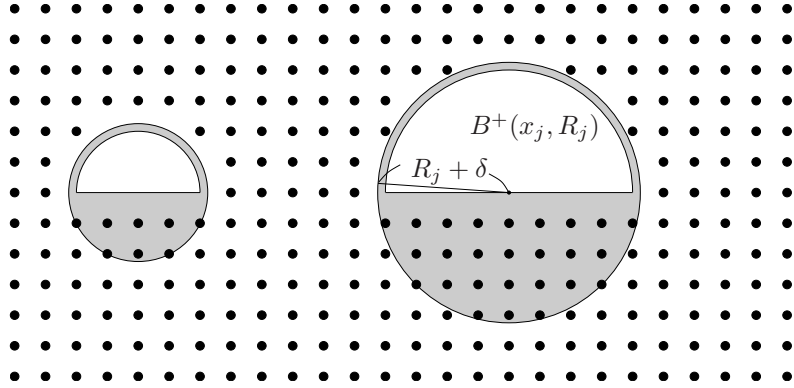}
\put(172,93){$B^+(x_j,R_j)$}
\put(150,77){$R_j+\delta$}
\end{overpic}
\caption{No dichotomy holds for $X=\Rd\sm \union_{j=1}^\infty B^+(x_j,R_j)$
with $E$ being the union of all small black balls.}
\label{fig:nodicho}
\end{figure}

\begin{proof}
From the construction, the balls
$\bigl\{B\bigl(x_j,\frac354^j\bigr)\bigr\}_{j=1}^\infty$ are pairwise disjoint.
To prove~\ref{k-unif-appr} let $\rho>0$ and 
$2^j\ge4\max\{\tau,\rho\}$.
Then
$B(x_j,\tau(R_j+\rho))\subset B\bigl(x_j,\frac354^j\bigr)$, and thus $B(x_j,\tau(R_j+\rho))$
does not intersect any of the balls 
$B(x_k,R_k)$, $k \ne j$.
Hence
\begin{align}   \label{eq:BXxjRj}
\capapED{\BX(x_j,R_j)}{\BX(x_j,\tau R_j)}
&\le \capapEDRn{B^-(x_j,R_j)}{B(x_j,\tau R_j)},
\nonumber \\
\capapED{\BX(x_j,R_j+\rho)}{\BX(x_j,\tau(R_j+\rho))}
&= \capapEDRn{B(x_j,R_j+\rho)}{B(x_j,\tau(R_j+\rho))},
\end{align}
which, together with 
translation and dilation for $\capapRn$, yields
\begin{align}
\frac{\capapED{\BX(x_j,R_j)}{\BX(x_j,\tau R_j)}}
{\capapED{\BX(x_j,R_j+\rho)}{\BX(x_j,\tau(R_j+\rho))}}
&\le
\frac{R_j^{\di-p}\capapEDRn{B^-(0,1)}{B(0,\tau)}}
{(R_j+\rho)^{\di-p}\capapEDRn{B(0,1)}{B(0,\tau)}} \nonumber \\
&\le
\frac{\capapEDRn{B^-(0,1)}{B(0,\tau)}}{\capapEDRn{B(0,1)}{B(0,\tau)}}
=:C
<1. \label{eq-cp-quot}
\end{align}
Thus~\ref{k-unif-appr} follows.

For the proof of \ref{k-dich}, let $0 < \de \le \tfrac{1}{4}$ and
note that if $x\in X$, then there exists $x'\in X\setm H$ such that
$d(x,x')\le\tfrac12$.
Now, by going at most length 1 in each of the coordinate directions, 
we can find $z\in\Z^n\cap (X\setm H)$ such that $d(x',z)\le\sqrt{n}$.
It thus follows from \reflem{capball} that
\begin{multline*}
\capap(E\cap B_{X}(x,2\sqrt\di),B_{X}(x,2\tau\sqrt\di))
\ge\capapED{B_{X}(z,\delta)}{B_{X}(z,4\tau\sqrt\di)} \\
  \ge C' \capap\bigl(\BX{(z,\delta)},{\BX\bigl(z,\tfrac{1}{2}\bigr)}\bigr)
= C' \capapRn\bigl({B(z,\delta)},{B\bigl(z,\tfrac{1}{2}\bigr)}\bigr)
\ge C'' \delta^{\di-p},
\end{multline*}
where $C'$ and $C''$ depend only on $n$, $p$ and $\tau$.
Taking infimum over $x\in X$, we obtain~(a).
It then follows from \reflem{dens=densin} and \refthm{main} that
\[
\liminf_{R\to\infty} \dens{R}{\tau}{E}
\ge {C'''}
\liminf_{R\to\infty} \densin{R}{\tau}{E}
={C'''} >0,
\]
where $C'''$ depends only on $n$, $p$ and $\tau$.
By \refeq{E=BzdH} we have
\[
E\cap\BX(x_j,R_j+\delta)\subset B^-(x_j,R_j+\delta).
\]
Moreover, if $2^j \ge 4\tau$, then 
\eqref{eq:BXxjRj} with $\rho=\delta$ yields
as in \eqref{eq-cp-quot},
\[
\frac
{\capapED{E\cap\BX(x_j,R_j+\delta)}{\BX(x_j,\tau(R_j+\delta))}}
{\capapED{\BX(x_j,R_j+\delta)}{\BX(x_j,\tau(R_j+\delta))}}
\le
\frac{\capapEDRn{B^-(0,1)}{B(0,\tau)}}{\capapEDRn{B(0,1)}{B(0,\tau)}} <1. 
\]
Hence 
\[
\dens{R_j+\delta}{\tau}{E}
\le
\frac{\capapEDRn{B^-(0,1)}{B(0,\tau)}}{\capapEDRn{B(0,1)}{B(0,\tau)}}
<1,
\]
so that
\(
\liminf_{R\to\infty} \dens{R}{\tau}{E}<1.
\)
Thus (b) is proved.
\end{proof}

The following example shows that
the Sobolev capacity $\Cp$ has no dichotomy nor a weak dichotomy similar 
to the one in \refthm{main-weak}.
Define
\[
\densCp{r}{E}
=\inf_{x\in X}
\frac{\Cp(E\cap B(x,r))}{\Cp(B(x,r))}.
\]
We are interested in the  behavior of $\densCp{r}{E}$
 as $r \to \infty$.

\begin{example} \label{exa:Cp}
Let $X=\R^n$ (unweighted) and $1<p<\infty$.
Note that $\mu(E)\le\Cp(E)$ for every measurable set $E$.  
For $B(x,r)$ and $r\ge1$ we can test the capacity 
with $u(y)=(1-\dist(y,B(x,r)))_+$, which shows that
\begin{equation}   \label{eq-Cp-Bxr}
    r^n \om_n 
    = \mu(B(x,r)) 
    \le \Cp(B(x,r)) 
    \le 2 \cdot (2r)^n \om_n
    = 2^{n+1} r^n \om_n,
\end{equation}
where $\om_n=\mu(B(0,1))$.
Let $M \ge 10$, $A=(M\Z)^n=\{\ldots,-M,0,M,\ldots\}^n$
and $E_{M}=\bigcup_{z \in A} B(z,1)$.  
Also let $x \in X$.

Using~\eqref{eq-Cp-Bxr} with $r=1$  and estimating
the number of balls $B(z,1)$  in $E_{M} \cap B(x,r)$, $r \ge 10M$, gives
\[
   \biggl(\frac{r}{M \sqrt{n}}\biggr)^n \om_n
     \le \mu(B(x,r) \cap E_{M}) 
 \le \Cp(B(x,r) \cap E_{M}) 
   \le    \biggl(\frac{3r}{M }\biggr)^n  2^{n+1}\om_n
     = 2 \biggl(\frac{6r}{M }\biggr)^n  \om_n.
\]
Combining this estimate with~\eqref{eq-Cp-Bxr}
shows that
\[
     \frac{1}{2(2M\sqrt{n})^n} \le  \liminf_{r \to \infty} \densCp{r}{E_{M}}
      \le \limsup_{r \to \infty} \densCp{r}{E_{M}}
      \le 2\biggl(\frac{6}{M }\biggr)^n.
\]
It follows that, by varying $M$,
$\liminf_{r \to \infty} \densCp{r}{E_M}$ can take at least a countable
number of different values in the interval $[0,1]$, including the end points
since $\densCp{r}{X}=1$ and $\densCp{r}{\emptyset}=0$ for all $r$.
Most likely it can take any value in the interval.
\end{example}

\section{Dichotomy and capacitarily stable collections}
\label{sect-cap-stable}

In studying the proof of Theorem~\ref{thm:main}, it turns out that
dichotomy holds for many more families
of sets than the family
of inner balls.
In this section we first extract the key properties such a family
might
have and then  demonstrate dichotomy 
under these assumptions.
We then proceed to give examples of such capacitarily stable families.

\begin{deff}   \label{def-stable}
A collection $\CC$ of bounded open subsets of $X$ is \emph{capacitarily stable}
if there exist constants $\tau>1$, $\ga \ge1$
and a function 
${\phi}:(0,\infty)^2\to(0,1]$ such that:
\begin{enumerate}
\renewcommand{\theenumi}{\text{(\roman{enumi})}}%
\renewcommand{\labelenumi}{\theenumi}%
\item \label{c-1}
For every ball $B\subset X$ we can find $U\in\CC$ such that $B\subset U\subset\gamma B$.

\item \label{c-2}
For each $U\in\CC$ there exists a ball $B_U\subset X$ such that
$B_U\subset U\subset\gamma B_U$.

\item \label{c-3}
For every $\rho,R>0$ and every $U\in\CC$ with $\diam(U)\ge R$ we have
\[
\frac{\cp(U_\rho,\UU)}{\cp(U,\UU)} \ge 1-\pip(\rho,R),
\]
where 
$U_\rho$ is the $\rho$\,-interior of $U$ as in~\eqref{eq:U_eps},
and $\UU:=\tau\gamma B_U$.

\item \label{c-4}
For every $\rho>0$,
\[
\lim_{R\to\infty} \pip(\rho,R)=0.
\]
\end{enumerate}
\end{deff}

\begin{deff}  \label{def-DCC}
Given a capacitarily stable collection $\CC$ with parameters
$\tau$, $\ga$ and $\pip$, we set for $r>0$  
and $E\subset X$, 
\[
\DCC(r,E) = 
\inf_{\substack{U\in\CC\\ r\le\diam(U)\le 2\ga r}}
\frac{\capapED{E\cap U}{\UU}}{\capapED{U}{\UU}}.
\]
\end{deff}

Note that since $X$ (under our assumptions) is connected and unbounded
we have that $r \le \diam(B(x,r)) \le 2r$ for every ball $B(x,r)$.
Hence, because of
\ref{c-1}, the collection $\{U\in\CC:r\le\diam(U)\le 2\ga r\}$
is nonempty, and thus $\DCC(r,E)<\infty$ (and so $\le1$).
A capacitarily stable collection $\CC$ might be associated with more than
one choice of the parameters $\tau$ and $\ga$. Different choices of $\tau$ and $\gamma$
impact the value of $\DCC(r,E)$. However, the value of $\DCC(r,E)$ is independent of 
the choice of $\phi$.

We are now ready to obtain the main dichotomy result for capacitarily
stable collections.
Since $X$ is unbounded it follows from Definition~\ref{def-stable}\,\ref{c-1}
that $\sup_{U \in \CC} \diam(U)=\infty$ whenever $\CC$ is a capacitarily
stable collection,
and thus it makes sense to consider the limits $R \to \infty$ in 
\refthm{meta} and \refcor{meta} below.

\begin{thm}  \label{thm:meta} 
Let $\CC$ and $\CC'$ be capacitarily stable collections of
bounded open sets in $X$, $\tau > 1$, and $E\subset X$.
Then the following statements are equivalent:
\begin{enumerate}
\renewcommand{\theenumi}{\textup{(\alph{enumi})}}%
\renewcommand{\labelenumi}{\theenumi}%
\item \label{meta-a}
$\DCC(r,E)>0$ for some $r>0$,
\item \label{meta-b1}
$\lim_{R\to\infty}\mathcal{D}^{\CC}(R,E)=1$,
\item \label{meta-b}
$\lim_{R\to\infty}\mathcal{D}^{\CC'}(R,E)=1$,
\item \label{meta-c} 
$\lim_{R\to\infty}\densin{R}{\tau}{E}=1$.
\setcounter{saveenumi}{\value{enumi}}
\end{enumerate}
\end{thm}

As an immediate corollary we obtain the following dichotomy.

\begin{cor} \label{cor:meta}
Let $\CC$ be a capacitarily stable collection of
bounded open sets in $X$. 
Then for every $E\subset X$ one of the following statements holds:
\begin{enumerate}
\item $\lim_{R\to\infty}\mathcal{D}^{\CC}(R,E)=0$,
\item $\lim_{R\to\infty}\mathcal{D}^{\CC}(R,E)=1$.
\end{enumerate}
Furthermore, these two possibilities  are independent of\/ 
$\CC$ and its associated parameters.
\end{cor}

Note also that by appealing to \refthm{main} we can directly obtain
several further statements equivalent to those in \refthm{meta}.

For the dichotomy to hold what happens at small scales is irrelevant. 
We could therefore have associated
yet another parameter $R_0 \ge 0$ with capacitarily stable collections,
requiring \ref{c-1} and \ref{c-3} in Definition~\ref{def-stable}
to hold only for $\diam(B)>R_0$ resp.\ $R>R_0$.
The implications 
\ref{meta-a} $\imp$ \ref{meta-b1}, \ref{meta-b},  \ref{meta-c}
in Theorem~\ref{thm:meta}
would then hold
provided that $r$ is sufficiently large (depending on $R_0$). 
A drawback would however have been that
here, as well as in results similar to Theorems~\ref{thm-cork-imp-stable}
and~\ref{thm:stable-2}, one also would have to consider
possible enlargements of this parameter.
We have refrained from this generalization.

\begin{proof}[Proof of \refthm{meta}]
To facilitate the proof we introduce one more statement that will
be shown to be equivalent to the statements in the theorem:
\begin{enumerate}
\renewcommand{\theenumi}{\textup{(\alph{enumi})}}%
\renewcommand{\labelenumi}{\theenumi}%
\setcounter{enumi}{\value{saveenumi}}
\item \label{meta-d}
$\dens{r}{\Lambda}{E}>0$ for some $r>0$.
\end{enumerate}
Recall that $\Lambda=100\lambda$, where $\lambda\ge1$ is 
the dilation constant in the \p-Poincar\'e inequality.

\ref{meta-b1} $\imp$ \ref{meta-a}
This is trivial.

\ref{meta-a} $\imp$ \ref{meta-d}
It is sufficient to prove that for all $r>0$,
\begin{equation}\label{eq:DCC<dens}
C\DCC(r,E)
\le \dens{\gamma r}{\Lambda}{E},
\end{equation}
where $C>0$ depends only on the parameters of $\CC$.
Let $r>0$ 
and $x\in X$.
By Definition~\ref{def-stable}\,\ref{c-1} we find $U\in\CC$ 
such that $B(x,r)\subset U\subset B(x,\gamma r)$,
and then by Definition~\ref{def-stable}\,\ref{c-2} 
we find a ball $B_U=B(x_U,r_U)$ such that
$B_U\subset U\subset \ga B_U$.
As $x_U \in U \subset B(x,\ga r)$, we see that $d(x,x_U)< \ga r$,
so that $B(x,\Lambda\gamma r)\subset B(x_U,(\Lambda+1)\gamma r)$.
Similarly, $d(x,x_U) < \ga r_U$ and thus
$U^*= B(x_U,\tau \ga r_U) \subset B(x, (\tau+1)\ga r_U)$.
Hence \reflem{capball} shows that 
\[
\capapED{E\cap B(x,\gamma r)}{B(x,\Lambda\gamma r)}
\ge\capapED{E\cap U}{B(x_U,(\Lambda+1)\gamma r)}
\simeq \capapED{E\cap U}{\UU}
\]
and
\[
\capapED{U}{\UU}
\ge \capapED{B(x,r)}{B(x,(\tau+1)\gamma r_U)} 
  \simeq \capapED{B(x,\gamma r)}{B(x,\Lambda \gamma r)}.
\]
Since $r\le\diam(U)\le 2\gamma r$, we get that
\[
\frac{\capapED{E\cap B(x,\gamma r)}{B(x,\Lambda\gamma r)}}{\capapED{B(x,\gamma r)}{B(x,\Lambda \gamma r)}}
\ge C\frac{\capapED{E\cap U}{\UU}}{\capapED{U}{\UU}}
\ge C\DCC(r,E).
\]
Taking the infimum with respect to $x\in X$, we obtain \eqref{eq:DCC<dens}.

\ref{meta-d} $\imp$ \ref{meta-b1}
The proof of this implication 
is almost the same as the proof of \refthm{main}.
Let $r>0$ be such that $\dens{r}{\La}{E}>0$.
Note that this 
implies the hypothesis \refeq{C/C>eta} of Lemma~\ref{lem:hmp<onUk}
for some $0<\eta<1$.
Let $U\in\CC$ and let $\UU$ be as in Definition~\ref{def-stable}\,\ref{c-3}.
Corollary~\ref{cor:C/C>} then gives
\[
\frac{\capapED{E\cap U}{\UU}}{\capapED{U_{k\Lambda r}}{\UU}}
\ge (1-(1-C_0\eta^{1/(p-1)})^k)^{p-1},
\]
whenever
$U_{k\Lambda r}\ne\emptyset$.

Take an arbitrary positive number $\alpha<1$ and
find a positive integer $k$ such that
the right-hand side of the above inequality is greater than $\alpha$.
If $\diam (U) \ge R$ and $B_U=B(x_U,r_U)$ is as in 
Definition~\ref{def-stable}\,\ref{c-2}, then $R\le2\ga r_U$ and $x_U\in U_{k\La r}$
provided that $k\La r < r_U$.
Thus, $U_{k\La r}\ne\emptyset$ whenever $\diam (U)\ge R> 2 \ga k \Lambda r$.
Definition~\ref{def-stable}\,\ref{c-3} with $\rho=\ga k\Lambda r$
then yields that for $R>2\rho$,
\[
\DCC(R,E) \ge \inf_{\substack{U\in\CC\\\diam(U)\ge R}}
\frac{\capapED{E\cap U}{\UU}}{\capapED{U_{\rho}}{\UU}}
\frac{\capapED{U_{\rho}}{\UU}}{\capapED{U}{\UU}}
\ge \al (1-\pip(\rho,R)).
\]
Letting $R \to \infty$ and then $\alpha \to 1$ shows that
$\lim_{R\to\infty} \DCC(R,E)=1$,
by Definition~\ref{def-stable}\,\ref{c-4}.

\ref{meta-b} $\eqv$ \ref{meta-d}
As we have now shown that \ref{meta-b1} $\eqv$ \ref{meta-d}, 
swapping the roles of $\CC$ and $\CC'$ 
immediately yields~\ref{meta-b} $\eqv$ \ref{meta-d}.

\ref{meta-c} $\eqv$ \ref{meta-d}
This follows directly from \refthm{main}.
\end{proof}

Next, we will present several useful examples of capacitarily stable collections.

\begin{thm}   \label{thm-cork-imp-stable}
Assume that $\CC$ is a family of open subsets of $X$ which satisfies 
Definition~\ref{def-stable}\,\ref{c-1}
with $\ga \ge 1$, and that there exists $\be>0$ such that
every $U\in\CC$ satisfies the interior
corkscrew condition with parameters $\ka$, $0$ and $\be\diam(U)$.
Let $\tau>1$ and $\gat:=\max\{\ga,1/\kappa \be\}$.
Then there is a function $\phi$ such that 
$\CC$ is capacitarily stable with parameters $\tau$, $\gat$ and $\phi$.
\end{thm}

In view of Remark~\ref{rem-Bin-corkscrew}, it follows
in particular that the family $\CC$  of all inner balls
is capacitarily stable, however,
note that
the density $\DCC$  
is obtained by looking at inner balls of 
diameters between $r$ and $2\ga r$, while $\densinname$ is obtained by looking at inner balls of radius $r$; thus these
two numbers could be different for each $r>0$.

\begin{proof}
For $U\in\CC$, pick $x\in U$ and use the corkscrew condition to find a ball
\[
{B_U:=}B(z_U,\ka\be\diam(U))\subset U\cap B(x,\be\diam(U)).
\]
Then \ref{c-2} holds with $\gat$.

To prove \ref{c-3} and \ref{c-4}, 
let $\rho,R>0$,
set $R_2=\min\{\be,\tau-1\}R$ and let $j$
be the largest integer such that
\[
\La^j \le \frac{\ka R_2}{2\rho}.
\]
Given 
$U\in\CC$ with $\diam(U)\ge R$, Lemma~\ref{lem:approx.in}
with $\Om=\UU:=\tau \gat B_U$ 
implies that
\[
\frac{\cp(U_\rho,\UU)}{\cp(U,\UU)} \ge (1-\eta_{\ka/2}^{{j_0}})^{p-1}
\ge 1-\max\{1,p-1\}\eta_{\ka/2}^{{j_0}} =: 1-\pip(\rho,R),
\]
where $j_0=\max\{j,0\}$.
Since ${j_0}\ge {\log R/{\log\La}+a}$ for some constant $a$ depending on $\ka$,
$\be$, $\tau$, $\La$ and $\rho$,
this implies that for every fixed $\rho$, we have
$\pip(\rho,R)\to0$ as $R\to\infty$,
i.e.\ \ref{c-4} holds.
\end{proof}

\begin{thm} \label{thm:stable-2}
Let $\CC$ be a family of open sets satisfying \ref{c-1} and \ref{c-2} of
Definition~\ref{def-stable}.
Let $\beta>0$.
For each $U\in\CC$, set
\[
U^\be=\{x\in X:\distin(x,U)<\be\diam(U)\}.
\]
Then $\CC^\be:=\{U^\be:U\in\CC\}$ is capacitarily stable.
\end{thm}

\begin{proof} 
It can be shown as in the proof of Proposition~3.4 in \cite{MR2338101}
that each $U^\be$ satisfies the interior corkscrew condition with
parameters $\ka=1/3L$, $0$ and 
$\be\diam(U)/3$, where $L$ is the quasiconvexity constant.
Also, by \ref{c-1} for $\CC$, if 
$B=B(x,r)$, then there is $U \in \CC$ such
that $B \subset U \subset \ga B$.
Since 
\[
\diam(U^\be) \le (1+2\be)\diam(U) \le 
{(1+2\be)2\ga r =: \gat} r,
\]
we see that $B \subset U^\be \subset \gat B$.
Thus, \ref{c-1} holds for $\CC^\be$ as well and
Theorem~\ref{thm-cork-imp-stable} concludes the proof.
\end{proof}

\begin{rem}
A particularly well shaped collection $\CC^\be$ is obtained if
\[
  \CC=\{B(x,r): x\in X \text{ and }r>0\}.
\]
Then the ``almost balls'' $B^\be(x,r)$ satisfy
$B(x,r)\subset B^\be(x,r) \subset B(x,(1+2\be)r)$
and are thus closer in shape to ordinary balls than what inner balls are,
cf.\ \eqref{eq:BinB}.
By Theorems~\ref{thm:meta} and~\ref{thm:stable-2}, 
dichotomy holds for these ``almost  balls''.
If $X$ is geodesic, i.e.\ for inner balls, we have
$B^\be(x,r)=B(x,(1+2\be)r)$. 
\end{rem}

The inner balls are John domains, see Remark~\ref{rem-john} below.
It is therefore natural to study dichotomy for John domains.

\begin{deff}
For an open set $U$ we let $\disU(x)=\dist(x,X \setm  U)$.
Let $0<c_J\le1$.
We say that $U$ is a \emph{$c_J$-John domain} 
if there exists (a John center) $x_U\in U$
such that each $x\in U$ can be connected to $x_U$
by a rectifiable curve $\gamma\subset U$ with
\begin{equation}\label{eq:John}
c_J \ell(\gamma(x,y))\le \disU(y)
\qtext{for all } y\in\gamma,
\end{equation}
where $\gamma(x,y)$ is the subcurve
of $\gamma$ from $x$ to $y$.
Such a curve $\gamma$ will be referred to as a $c_J$-John curve connecting
$x$ and $x_U$.
\end{deff}

We next show that John domains satisfy the interior corkscrew condition.
(Note that the only unbounded John domain is $X$ itself which is excluded
from our considerations.)

\begin{lem}\label{lem:John>ics}
Let $U$ be a bounded  $c_J$-John domain with
$0<c_J\le 1$.
Then $U$ satisfies the interior corkscrew condition with
parameters $\kappa=c_J^2/4$, $0$ and $\diam(U)$.
\end{lem}

\begin{proof}
Let $x_U$ be the John center of $U$.
For any $x\in U$ we find a $c_J$-John curve $\gamma$ connecting $x$ and $x_U$,
i.e., \refeq{John} holds.
In particular, $\disU(x_U)\ge c_J\ell(\gamma)\ge c_J d(x,x_U)$.
Taking the supremum with respect to $x\in U$, we obtain
$\disU(x_U)\ge c_J\diam(U)/2$.

Now let $x\in U$ and $0<r<\diam(U)$.
We claim that $U\cap B(x,r)$ contains a ball of radius $c_J^2r/4$.
If $\disU(x)\ge c_J^2r/4$, then $U\cap B(x,r)\supset B(x,c_J^2r/4)$.
So, suppose that $\disU(x)< c_J^2r/4$.
Then
\[
\frac{c_Jr}2<\frac{c_J\diam(U)}2\le \disU(x_U)\le\disU(x)+d(x,x_U)<\frac{c_J^2r}4+d(x,x_U),
\]
so that $d(x,x_U)>c_Jr/4$.
Let $\gamma$ be a John curve connecting $x$ and $x_U$.
We find a point $y\in\gamma$ with $d(x,y)=c_Jr/4$.
Then $\disU(y)\ge c_J\ell(\gamma(x,y))\ge c_J^2r/4$.
Hence, 
\[
 B(y,c_J^2r/4) \subset U \cap B(x,c_J^2r/4+c_Jr/4) \subset  U\cap B(x,r),
\]
as required.
\end{proof}

\begin{thm} \label{thm:John}
Let $\Jclass(c_J)$ be the family of all bounded $c_J$-John domains.
If $\Jclass(c_J)$ satisfies Definition~\ref{def-stable}\,\ref{c-1},
in particular if $0<c_J\le1/L$
(where $L$ is the quasiconvexity constant), then
it is capacitarily stable.
\end{thm}

\begin{proof}
This follows directly from Lemma~\ref{lem:John>ics} and
Theorem~\ref{thm-cork-imp-stable}, together with the following remark.
\end{proof}

\begin{rem}     \label{rem-john}
It is easy to see that if $X$ is a geodesic space, 
then $B(x,r)$ is a $1$-John domain with John center $x$.
On the other hand, the counterexample
in \refthm{nodicho} is not geodesic 
with respect to $d$ and an ordinary ball $B(x,r)$ need not be a John domain.

Since $X$ is always geodesic with respect to $\din$, 
it follows that $\Bin(x,r)$ is a $1$-John domain with respect to $\din$,
and hence a ${(1/L)}$-John domain with respect to $d$.
This, in particular, means that for $c_J\le1/L$ the family $\Jclass(c_J)$
is nonempty. 
Furthermore, as $X$ is quasiconvex, the family $\Jclass(c_J)$
satisfies Definition~\ref{def-stable}\,\ref{c-1}.
\end{rem}

The following example shows that there are unbounded
metric spaces with no $c_J$-John domains of large diameter,
when $c_J$ is close to $1$.
Thus, $\Jclass(c_J)$ is not always capacitarily stable
since Definition~\ref{def-stable}\,\ref{c-1} is violated in such situations.

\begin{example}
Consider
\[
X=\bigl\{(x,y)\in\R^2 : |y-\cos x| \le \tfrac12\bigr\},
\]
equipped with the Euclidean metric and the $2$-dimensional Lebesgue 
measure, see Figure~\ref{fig:NoJohn}. 
Since
$X$ is biLipschitz equivalent to $[0,1]\times\R$, 
it follows that the measure on $X$ is doubling
and supports a $1$-Poincar\'e inequality.

\begin{figure}[htb]
\footnotesize
\begin{overpic}[scale=.9]
{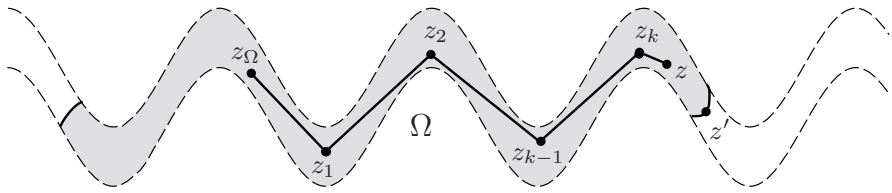}
\put(235,24){\normalsize$\Om$}
\put(168,53){$z_\Om$}
\put(198,10){$z_1$}
\put(240,61){$z_2$}
\put(273,15){$z_{k-1}$}
\put(319,61){$z_k$}
\put(334,46){$z$}
\put(347,23){$z'$}
\end{overpic}
\caption{No $c_J$-John domains of large diameter if $c_J>\pi/\sqrt{1+\pi^2}$.}
\label{fig:NoJohn}
\end{figure}

Let $\Om\subset X$ be a bounded domain with $\diam(\Om)>2\pi+3$ 
and let $z_\Om=(x_0,y_0)\in\Om$ act as a John center.
By translation of $\Om$, we can assume that $|x_0| \le \pi$.
Find $z'=(x',y')\in\bdry\Om$ so that $|x'|$ is as large as possible.
Because of the symmetry of $X$ and as
$\diam(\Om)>2\pi+3$, we can assume that 
$x'\in (k\pi,(k+1)\pi]$ for some integer $k\ge1$ since 
$\diam([-\pi,\pi]\times[-\frac32,\frac32])<2\pi+3$.

A simple geometric argument then shows that
$\de_\Om(z_\Om)\le |z_\Om|+|z'|\le (k+2)\pi+3$
and that any curve $\ga$ in $\Om$, which connects $z_\Om$ with a point 
$z=(x,y)\in\Om$, where 
$x > k \pi$, 
intersects vertical lines of $x$-coordinate $j\pi$, and hence
contains points $z_j=(j\pi,y_j)$ with $|y_j-\cos j\pi|\le\frac12$
for $j=1,\dots,k$.
Since $|y_j-y_{j+1}|\ge1$, we conclude that
\[
\ell(\ga)\ge \sum_{j=1}^{k-1} | z_j-z_{j+1}|
      \ge (k-1)\sqrt{1+\pi^2}.
\]
Since
\[
\frac{\de_\Om(z_\Om)}{\ell(\ga)} 
\le \frac{(k+2)\pi+3}{(k-1)\sqrt{1+\pi^2}} 
\to \frac{\pi}{\sqrt{1+\pi^2}} <1, \quad \text{as }k\to\infty,
\]
we see that for every $c_J> \pi/\sqrt{1+\pi^2}$, there exists $r>0$ such that
there are no $c_J$-John domains in $X$ with diameter at least $r$.
\end{example}

\end{document}